\documentclass[english,a4paper, abstract = true]{scrartcl}
\pdfoutput=1
\usepackage[utf8x]{inputenc}

\usepackage{microtype}
\usepackage[shortlabels]{enumitem}
\usepackage[mathscr]{euscript}
\usepackage{a4wide,amsmath,amssymb,bm,graphicx,hyperref}
\usepackage[english]{babel}
\usepackage{amsthm}
\usepackage[T1]{fontenc}
\usepackage{dsfont}
\usepackage{mathtools}
\usepackage{tabularx}
\usepackage{parskip}

\usepackage{geometry}
\usepackage{verbatimbox}

\usepackage{braket}

\usepackage{marginnote}

\usepackage[dvips]{epsfig}
\usepackage[bf]{caption}

\usepackage{adjustbox}
\usepackage{dirtytalk}
\usepackage{mdframed}

\usepackage{tikz}
\usetikzlibrary{matrix}
\usetikzlibrary{arrows.meta}
\usetikzlibrary{cd}
\usetikzlibrary{calc}
\usetikzlibrary{hobby}

{\begin{enumerate}[leftmargin={1.5}cm,rightmargin={1.2}cm,label={\small\bfseries #1}]
\item\begin{small}\sffamily}%
{\end{small}\end{enumerate}}

\newcommand{\subs}[1]{\ensuremath{_{\textrm{#1}}}}

\newcommand{\mtn}{\mathds N}%
\newcommand{\mtr}{\mathds R}%
\newcommand{\mtc}{\mathds C}%
\newcommand{\mth}{\mathds H}%
\DeclareMathOperator{\spn}{span}

\DeclareMathOperator{\Tr}{Tr}
\DeclareMathOperator{\End}{End}

\DeclarePairedDelimiter\norm{\lVert}{\rVert}
\DeclarePairedDelimiter\abs{\lvert}{\rvert}
\DeclarePairedDelimiter\expect{\langle}{\rangle}
\newcommand{\snab}[1]{\nabla_{\!#1}}%

\makeatletter
\let\oldabs\abs
\def\abs{\@ifstar{\oldabs}{\oldabs*}}
\let\oldnorm\norm
\def\norm{\@ifstar{\oldnorm}{\oldnorm*}}

\let\oldexpect\expect
\def\expect{\@ifstar{\oldexpect}{\oldexpect*}}
\makeatother

\theoremstyle{plain}
\newtheorem{lemma}{Lemma}[section]
\newtheorem{corollary}[lemma]{Corollary}
\newtheorem{proposition}[lemma]{Proposition}

\newtheorem{question}[lemma]{Question}
\newtheorem*{ntheorem}{Theorem}
\theoremstyle{remark}

\newtheorem*{remark}{Remark}
\theoremstyle{definition}
\newtheorem{defn}[lemma]{Definition}
\newtheorem{example}[lemma]{Example}

\theoremstyle{plain}
\newtheorem{theorem}{Theorem}
\newtheorem*{futuretheoreminner}{Theorem~\thefuturetheoreminner}
\newcommand{\thefuturetheoreminner}{} 

\ExplSyntaxOn

\prop_new:N \g_alevel_future_prop

\NewDocumentEnvironment{futuretheorem}{ m o +b }
 {
  \renewcommand{\thefuturetheoreminner}{\ref{#1}}
  \IfNoValueTF{#2}
   {\futuretheoreminner}
   {\futuretheoreminner[#2]}
  #3
 }
 {
  \endfuturetheoreminner
  \prop_gput:Nnn \g_alevel_future_prop { #1 } { #3 }
  \IfValueT{#2}{ \prop_gput:Nnn \g_alevel_future_prop { #1-attr } { #2 } }
 }

\NewDocumentCommand{\pasttheorem}{m}
 {
  \prop_if_in:NnTF \g_alevel_future_prop { #1-attr }
   {
    \begin{theorem}[\prop_item:Nn \g_alevel_future_prop { #1-attr }]
   }
   {
    \begin{theorem}
   }
  \label{#1}
  \prop_item:Nn \g_alevel_future_prop { #1 }
  \end{theorem}
 }

\ExplSyntaxOff

\newcommand{\defeq}{\vcentcolon=}

\usepackage{setspace}
\usepackage{cite}

\title{Polynomial harmonic morphisms and eigenfamilies on spheres}
\author{Oskar Riedler}
\date{\normalsize \today}
\begin{document}
\maketitle

\begin{abstract}
The eigenfamilies of Gudmundsson and Sakovich \cite{gudmundsson-sakovich-08} can be used to generate harmonic morphisms, proper $r$-harmonic maps, and minimal co-dimension $2$ submanifolds. This article begins by characterising the globally defined eigenfamilies of the sphere $S^m$; they correspond to orthogonal families of homogeneous polynomial harmonic morphisms from $\mtr^{m+1}$ to $\mtc$, all of the same degree. We investigate and construct such families, paying special attention to those that are not congruent to families of holomorphic polynomials. Strong restrictions for families of such polynomials are found in low dimensions, and the pairs of degree $2$ maps that induce an eigenfamily are classified.
\end{abstract} 

\section{Introduction}

Let $(M,g), (N,h)$ be two Riemannian manifolds. A \textit{harmonic morphism} is a smooth map $\varphi:M\to N$ that pulls back germs of harmonic maps on $N$ to germs of harmonic maps on $M$, i.e. for all (locally defined) harmonic maps $f:U\subseteq N\to\mtr$ one has that $f\circ\varphi:\varphi^{-1}(U)\subseteq M\to\mtr$ is harmonic. The theory of harmonic morphisms between Riemannian manifolds begins with the works of Fuglede \cite{fuglede-78} and Ishihara \cite{ishihara-79}, who found a geometric characterisation of the above condition: A map $\varphi:M\to N$ is a harmonic morphism if and only if it is harmonic and weakly horizontally conformal.

The case $\dim(N)=2$ is especially interesting for two reasons. First the level sets $\varphi^{-1}(\{y\})\cap (M\setminus S)$ are then minimal submanifolds of $M\setminus S$ (here $S=\{ x\in M\mid D_x\varphi=0\}$ denotes the singular points of $\varphi$
), and one obtains a codimension $2$ foliation of $M\setminus S$ consisting of minimal submanifolds, c.f. \cite{baird-eells-80}. Secondly the property of $\varphi:M\to N$ being a harmonic morphism is in this case invariant under conformal changes of the metric on $N$. So one can always assume, at least locally, that the codomain is $\mtc$ with its Euclidean metric.

For maps with codomain $\mtc$ the conditions of being harmonic and weakly horizontally conformal can be expressed by simple equations. A map $\varphi:M\to \mtc$ is a harmonic morphism if and only if (see e.g. \cite{baird-wood-book}):
\begin{gather}
\Delta \varphi = 0,\label{eq: hm1}\\
g_\mtc(\nabla \varphi, \nabla\varphi) =0. \label{eq: hm2}
\end{gather}
Here $\nabla\varphi$ is the section of $TM\otimes\mtc$ given by $\nabla\mathrm{Re}(\varphi)+i\nabla\mathrm{Im}(\varphi)$ and $g_\mtc$ denotes the complex bilinear extension of the metric $g$ to $TM\otimes\mtc$.

Equations (\ref{eq: hm1}) and (\ref{eq: hm2}) form and overdetermined system of partial differential equations, so a priori it is not clear if any solutions exist. In fact Baird and Wood \cite{baird-wood-02} have shown that the 3-dimensional homogeneous space $\mathrm{Sol}$ does not even admit any locally defined harmonic morphisms to $\mtc$.

Motivated by the existence problem of harmonic morphisms, Gudmundsson and Sakovich \cite{gudmundsson-sakovich-08} introduced the notion of an \textit{eigenfamily}, which can be used to generate families of harmonic morphisms:

\begin{defn}
Let $\lambda,\mu\in\mtc$. A family $\mathcal F$ of smooth functions from $M$ to $\mtc$ is called a \textit{$(\lambda,\mu)$-eigenfamily} on $M$ if for all $\varphi,\psi\in\mathcal F$:
\begin{gather}
\Delta \varphi = \lambda \varphi,\label{eq: ef1}\\
g_\mtc(\nabla \varphi, \nabla \psi)= \mu\, \varphi\psi. \label{eq: ef2}
\end{gather}
Functions that are elements of some $(\lambda,\mu)$-eigenfamily are called \textit{$(\lambda,\mu)$-eigenfunctions}.
\end{defn}

\begin{remark}
Comparing with equations (\ref{eq: hm1}), (\ref{eq: hm2}) one notes that a $(0,0)$-eigenfunction is the same as a complex valued harmonic morphism.
\end{remark}

\begin{ntheorem}[Gudmundsson-Sakovich, c.f. Theorem 2.5 in \cite{gudmundsson-sakovich-08}]
Let $\mathcal F$ be a $(\lambda,\mu)$-eigenfamily on $M$, then for any $\varphi_1,...,\varphi_k\in\mathcal F$ and homogeneous polynomials $P,Q\in\mtc[z_1,...,z_k]$ of the same degree, the map $\dfrac{P\circ (\varphi_1,...,\varphi_k)}{Q\circ (\varphi_1,...,\varphi_k)}$ is a harmonic morphism on $M\setminus\{x\in M\mid Q(\varphi_1(x),...,\varphi_k(x))=0\}$.
\end{ntheorem}

In a series of papers \cite{gudmundsson-sakovich-08, gudmundsson-siffert-sobak-22, ghandour-gudmundsson-explicit23} Gudmundsson and co-authors construct eigenfamilies on each of the classical compact symmetric spaces. A duality principle induces eigenfamilies on the non-compact duals of these spaces. In this way they explicitly construct for many symmetric spaces the first known  examples of harmonic morphisms.

In addition to their use in generating harmonic morphisms, it was shown by Gudmundsson and Sobak \cite{gudmundsson-sobak-20} that eigenfunctions can be used to explicitly construct proper $r$-harmonic maps with an open dense domain, and Gudmundsson and Munn \cite{gudmundsson-munn-23} use them to construct minimal submanifolds of co-dimension $2$.

In this article we begin by considering eigenfamilies on the sphere $S^m$ with global domain. It is a simple observation that such families are in bijection with $(0,0)$-eigenfamilies on $\mtr^{m+1}$ consisting of homogeneous polynomials all of the same degree. 

\begin{futuretheorem}{thm: Sn-eigen}
Let $\mathcal F$ be a family of maps from $S^m$ to $\mtc$, and $\lambda,\mu\in \mtc$. The following are equivalent:
\begin{enumerate}[label={(\roman*)}]
\item $\mathcal F$ is a $(\lambda,\mu)$-eigenfamliy.
\item There is a $(0,0)$-eigenfamily $\widetilde{\mathcal F}$ of homogeneous polynomials from $\mtr^{m+1}$ to $\mtc$, all of the same degree $d$, so that the map
$$\widetilde{\mathcal F}\to\mathcal F,\qquad F\mapsto F\lvert_{S^m},$$
is a well defined bijection and $\lambda =- d(d+m-1)$, $\mu=-d^2$.
\end{enumerate}
\end{futuretheorem}

The other compact rank one symmetric spaces, with exception of the Cayley plane, can be obtained from the sphere via a Riemannian submersion with totally geodesic fibres. One then gets a characterisation of the global eigenfamilies on the real, complex, and quaternionic projective spaces:

\begin{futuretheorem}{thm: CROSS}
Let $X\in\{\mathds{RP}^m, \mathds{CP}^m, \mathds{HP}^m\}$ and suppose $\mathcal F$ is a family of functions on $X$ to $\mtc$. Denoting with $\pi: S^{m(X)}\to X$ the usual submersion one has that the following are equivalent:
\begin{enumerate}[label=(\roman*)]
\item $\mathcal F$ is a $(\lambda,\mu)$-eigenfamily on $X$.
\item $\pi^*(\mathcal F)=\{F\circ\pi : S^{m(X)}\to\mtc\mid F\in\mathcal F\} $ is a $(\lambda,\mu)$-eigenfamily on $S^{m(X)}$.
\end{enumerate}
\end{futuretheorem}

The rest of the article then concerns $(0,0)$-eigenfamilies on $\mtr^{m+1}$ consisting of homogeneous polynomials of the same degree. The elements of such an eigenfamily are homogeneous polynomial harmonic morphisms, which are objects of independent interest and have been studied in the past (see e.g. \cite{ABB-99} and \cite{baird-wood-book}). Note that if a family of polynomials is a $(\lambda,\mu)$-eigenfamily then clearly $(\lambda,\mu)=(0,0)$, so we will refer to $(0,0)$-eigenfamilies simply as eigenfamilies when dealing with polynomials.

Families of holomorphic maps $\mtc^m\to\mtc$ are automatically $(0,0)$-eigenfamilies. In particular the space of holomorphic homogeneous polynomials of degree $d$ provides an immediate example of a polynomial eigenfamily; that these induce eigenfamilies on the sphere is known (see \cite{gudmundsson-munn-23}). The question of finding families that are not holomorphic with respect to any K\"ahler structure then arises naturally. The following definition is a convenient way of encoding holomorphy of functions (also on odd dimensional $\mtr^m$), without referencing any fixed K\"ahler structure:

\begin{defn}\begin{enumerate}
\item A family $\mathcal F$ of functions from $\mtr^m\to\mtc$ said to be \textit{uniformly of complex type} if there is a $k\geq0$, an orthogonal projection $\pi:\mtr^m\to\mtc^k$, and a family $\mathcal G$ of holomorphic maps $\mtc^k\to\mtc$ so that $\mathcal F = \pi^*(\mathcal G) \defeq \{G\circ\pi \mid G\in\mathcal G\}$.
\item A function $F:\mtr^m\to\mtc$ is said to be of \textit{complex type} if the family $\{F\}$ is uniformly of complex type.
\end{enumerate}
\end{defn}

By work of Baird and Wood \cite{baird-wood-88} and Wood \cite{wood-92} it is well known that any globally defined harmonic morphism $\mtr^m\to\mtc$ for $m\leq 4$ is of complex type. This remains true in the context of $(0,0)$-eigenfamilies:

\begin{futuretheorem}{thm: 4-ct}
Let $m\leq 4$ and suppose $\mathcal F$ is a $(0,0)$-eigenfamily of maps (not necessarily polynomials) $\mtr^m\to\mtc$. Then $\mathcal F$ is uniformly of complex type.
\end{futuretheorem}

There is a very simple criterion with which one can determine when a function or family is of complex type:

\begin{futuretheorem}{thm: complex-type}
A family $\mathcal F$ of smooth functions from $\mtr^m$ to $\mtc$ is uniformly of complex type if and only if the gradients $\{\snab{x}F\mid x\in\mtr^m, F\in\mathcal F\}$ span an isotropic subspace of $\mtr^m\otimes\mtc$, i.e. if and only if for all $F,G\in\mathcal F$ and $x,y\in\mtr^m$:
\begin{equation}
(\snab{x} F)^T\, \snab{y}G =0.\label{eq: hol}
\end{equation}
\end{futuretheorem}
We expect this result to be known. But since we do not have a reference for this characterisation, a proof is included in Section \ref{sec: hol} for convenience of the reader. It is interesting to note that for a homogeneous degree $2$ polynomial $\mtr^m\to\mtc$, $x\mapsto x^TAx$ the condition of being a harmonic morphism and the condition of being of complex type coincide.

Using Theorem \ref{thm: complex-type}, one can check for example that the following family is an eigenfamily on $\mtr^8$ consisting of degree $3$ homogeneous polynomials, and that the family is not uniformly of complex type:
$$\{\mtc^4\to\mtc, (z,u,w,v)\mapsto a(z^2w+zu\overline v)+b(zu\overline w-z^2v)+c(u^2\overline v+zuw)+d(u^2\overline w-zuv)\mid a,b,c,d\in\mtc\}.$$

A feature of this example is that even though the maps are generically not of complex type, they are still all holomorphic in the variables $z,u$. It turns out to be useful to encode the existence of such holomorphic variables in the notion of an \textit{axis of holomorphy} (c.f. Definition \ref{def: axis}). One has for instance:

\begin{futuretheorem}{thm: 56-reduc}
Let $m\in\{5,6\}$ and suppose $\mathcal F$ is an eigenfamily of homogeneous polynomial harmonic morphisms from $\mtr^m$ to $\mtc$. If $\mathcal F$ has a uniform axis of holomorphy of (real) dimension at least $2$, then $\mathcal F$ is uniformly of complex type.
\end{futuretheorem}

Eigenfamilies consisting of homogeneous degree $2$ polynomials on $\mtr^m$ always admit a uniform axis of holomorphy of dimension at least $2$:
\begin{futuretheorem}{thm: deg2-axis}
Let $\mathcal F$ be an eigenfamily of homogeneous degree $2$ polynomial harmonic morphisms $\mtr^m\to\mtc$. Then $\mathcal F$ admits a uniform axis of holomorphy of (real) dimension at least $\min(2,m)$.
\end{futuretheorem}

Leveraging the notion of an axis of holomorphy allows one to classify all pairs $F_1,F_2$ of degree $2$ homogeneous polynomials $\mtr^m\to\mtc$ for which $\{F_1,F_2\}$ is an eigenfamily. See Definition \ref{def: deg2-data} in Section \ref{sec: deg2} for the relevant definitions:

\begin{futuretheorem}{thm: deg2-eigenpairs}
Let $F_1, F_2:\mtr^m\to\mtc$ be two homogenous degree $2$ polynomials so that $\{F_1,F_2\}$ is a full eigenfamily. Then there are subspace data $(n,k,\delta)$ as well as polynomial and twisting data $(P_1,P_2,A)$, $(Y,C,v)$ so that:
\begin{enumerate}
\item Up to an isometry of the domain one can decompose $\mtr^m\cong\mtc^n\oplus\mtc^k\oplus\mtr^\delta$ so that $\mtc^n$ is a maximal uniform axis of holomorphy for $\{F_1,F_2\}$.
\item Let $X=(\frac12vv^T+C)Y^{-1}$. One has, with respect to the above decomposition:
\begin{align*}
F_1 ((z_1,....,z_n), (w_1,...,w_k),t) &= P_1(z_1,...,z_n)+\sum_{ij}z_iA_{ij}w_j\\
F_2 ((z_1,....,z_n), (w_1,...,w_k),t) &= P_2(z_1,...,z_n)+\sum_{ij}z_iA_{ij}\left(\sum_lX_{jl}w_l+\sum_lY_{jl}\overline{w}_l+v_jt\right)
\end{align*}
\end{enumerate}
Additionally for any subspace type together with polynomial and twisting data the above equations define an eigenfamily on $\mtr^{2(n+k)+\delta}$.
\end{futuretheorem}

The article ends with examples of methods with which homogeneous polynomial eigenfamilies on $\mtr^m$ can be constructed, see Section \ref{sec: constructions}.

The article is organised as follows: In Section \ref{sec: eigen-spheres} we provide background information on harmonic morphisms and $(\lambda,\mu)$-eigenfamilies, and prove Theorems \ref{thm: Sn-eigen} and \ref{thm: CROSS}. In Section \ref{sec: hol} we investigate the question of when $(0,0)$-eigenfamilies are holomorphic, here we define the notions of complex type and that of an axis of holomorphy and provide examples of eigenfamilies that are not of complex type. Theorems \ref{thm: complex-type}, \ref{thm: 4-ct}, \ref{thm: 56-reduc} are proven in this section, and we describe how an axis of holomorphy leads to a reduction principle (c.f. Proposition \ref{prop: reduction-axis}). In Section \ref{sec: deg2} we investigate the case of eigenfamilies consisting of degree $2$ polynomials, showing Theorems \ref{thm: deg2-axis} and \ref{thm: deg2-eigenpairs}. Finally Section \ref{sec: constructions} reviews some constructions.

The author gratefully acknowledges the support of Germany's Excellence Strategy EXC 2044 390685587, Mathematics M\"unster: Dynamics-Geometry-Structure.

\section{Eigenfamilies on spheres}\label{sec: eigen-spheres}

In subsection \ref{subsec: hm-ef} we review some basic facts about harmonic morphisms and $(\lambda,\mu)$-eigenfamilies. In subsection \ref{subsec: sn} we show Theorem \ref{thm: Sn-eigen} and we show Theorem \ref{thm: CROSS} in subsection \ref{subsec: cross}.

For an overview of the theory and history of harmonic morphisms we recommend both the book \cite{baird-wood-book} and the regularly updated online bibliography \cite{hm-bib}.

\subsection{Harmonic morphisms and $(\lambda,\mu)$-eigenfamilies}\label{subsec: hm-ef}

Let $(M,g)$, $(N,h)$ be two Riemannian manifolds.

\begin{defn}
A smooth map map $\varphi:M\to N$ is called a \textit{harmonic morphism} if for every $U\subseteq N$ open and $f:U\to \mtr$ harmonic the composition $f\circ \varphi: \varphi^{-1}(U)\to \mtr$ is again harmonic.
\end{defn}

These maps were originally studied in the context of Brelot harmonic spaces and potential theory. The geometric theory begins with the work of Fuglede and Ishihara, who independently found an infinitesimal characterisation of the above definition on Riemannian manifolds:

\begin{ntheorem}[Fuglede \cite{fuglede-78}, Ishihara \cite{ishihara-79}]
A smooth map $\varphi:M\to N$ is a harmonic morphism if and only if its harmonic and weakly horizontally conformal.
\end{ntheorem}

We briefly explain these conditions:

\begin{defn}\label{def: whc}
\begin{enumerate}
\item A smooth map between two Riemannian manifolds $(M,g)$ and $(N,h)$ is called \emph{harmonic} if it is a stationary point of the energy functional $\varphi\mapsto \int_M \|D\varphi\|^2$.
\item A smooth map $\varphi:(M,g)\to(N,h)$ is called \emph{weakly horizontally conformal} if at any point $x\in M$ the differential $D_x\varphi$ is either $0$ or it maps the horizontal space $\ker(D_x\varphi)^\perp\subset T_xM$ conformally onto $T_{\varphi(x)}N$.
\end{enumerate}
\end{defn}
\begin{remark}
\begin{enumerate}
\item $\varphi:(M,g)\to (N,h)$ is harmonic if and only if the tension field $\tau(\varphi)=\Tr(\nabla D\varphi)$ vanishes.
\item If $\varphi:(M,g)\to (N,h)$ is weakly horizontally conformal there exists a function $\lambda:M\to\mtr$, called the \emph{conformality factor}, so that $\varphi^*(h)=\lambda^2 g\lvert_{\ker(d\varphi)^\perp}$, meaning that for all $x\in M$, $v,w\in \ker(D_x\varphi)^\perp\subset T_xM$ one has:
$$h(D_x\varphi\, v, D_x\varphi\, w) =\lambda^2(x)\,g(v,w).$$
\end{enumerate}
\end{remark}

The theory of harmonic morphisms between Riemannian manifolds is extremely rich, we will only briefly remark on two aspects. First one has that homogeneous polynomial harmonic morphisms between vector spaces arise naturally: 

\begin{ntheorem}[Fuglede \cite{fuglede-78}]
Let $\varphi: M\to N$ be a harmonic morphism and suppose $x\in M$ is a critical point of $\varphi$. Then $x$ is of finite order and the symbol of $\varphi$ at $x$ is a homogeneous polynomial harmonic morphism between the real vector spaces $T_xM\to T_{\varphi(x)}N$.
\end{ntheorem}

Secondly there is a strong connection with minimal submanifolds, in the case $\dim(N)=2$ this connection is especially strong and is expressed by the following statement:

\begin{ntheorem}[Baird-Eells \cite{baird-eells-80}]
Suppose $\varphi:M\to N$ is horizontally weakly conformal and $\dim(N)=2$, then $\varphi$ is harmonic (and so a harmonic morphism) if and only if the fibres of $\varphi$ are minimal at regular points.
\end{ntheorem}

In the case that $\dim(N)=2$ one notes that the property of being a harmonic morphism is preserved under conformal transformations of the codomain. In particular if $\dim(N)=2$ the problem of local existence of harmonic morphisms $M\to N$ is captured entirely by the case $N=\mtc$ equipped with its Euclidean metric. A calculation shows that a map $\varphi:M\to\mtc$ is a harmonic morphism if and only if (see e.g. \cite{fuglede-78}):
$$\Delta\varphi =0,\qquad g_\mtc(\nabla \varphi, \nabla \varphi)=0.$$
Here $\Delta$ denotes the Laplacian of $(M,g)$, $g_\mtc$ is the complex bilinear extension of $g$ to $TM\otimes\mtc$, and $\nabla \varphi:M\to TM\otimes\mtc$ is the map $\nabla\mathrm{Re}(\varphi)+i\,\nabla\mathrm{Im}(\varphi)$.

We now turn to eigenfamilies. It is often convenient to refer to the \textit{conformality operator} in this context:

\begin{defn}
The \textit{conformality operator} $\kappa:C^\infty(M;\mtc)\times C^\infty(M;\mtc)\to C^\infty(M;\mtc)$ is defined by $\kappa(\varphi,\psi)=g_\mtc(\nabla \varphi,\nabla \psi)$.
\end{defn}

A map $\varphi:M\to\mtc$ is then weakly horizontally conformal if and only if $\kappa(\varphi,\varphi)=0$. Recall the definition of a $(\lambda,\mu)$-eigenfamily from the introduction:

\begin{defn}[Gudmundsson-Sakovich \cite{gudmundsson-sakovich-08}]
Let $\lambda,\mu\in\mtc$. A family $\mathcal F$ of smooth functions from $M$ to $\mtc$ is called a \textit{$(\lambda,\mu)$-eigenfamily} if for all $\varphi,\psi\in\mathcal F$:
$$\Delta \varphi = \lambda \varphi, \qquad \kappa(\varphi,\psi)=\mu\varphi\psi.$$
Functions that are elements of some $(\lambda,\mu)$-eigenfamily are called \textit{$(\lambda,\mu)$-eigenfunctions}.
\end{defn}
As explained in the introduction, eigenfamilies can provide a wealth of harmonic morphisms \cite{gudmundsson-sakovich-08}, proper $r$-harmonic maps \cite{gudmundsson-sobak-20}, and mininal submanifold of codimension 2 \cite{gudmundsson-munn-23}.

It is useful to note that the property of being an eigenfamily is closed under taking powers:
\begin{remark}If $\mathcal F$ is a $(\lambda,\mu)$-eigenfamily and $d\in\mtn$, then
$$\mathcal F^d\defeq \spn_\mtc\{\varphi_1\cdots \varphi_d \mid \varphi_i\in\mathcal F\}$$
is a $(\widetilde \lambda,\widetilde\mu)$-eigenfamily, where $\widetilde\lambda = d(\lambda+(d-1)\mu)$ and $\widetilde\mu = d^2\mu$. See e.g. Lemma A.1 in \cite{gudmundsson-sakovich-08}.
\end{remark}

\subsection{Eigenfamilies on $S^m$}\label{subsec: sn}
We now show the first result of this article, which is a characterisation of the globally defined eigenfamilies on the spheres $S^m$.

\pasttheorem{thm: Sn-eigen}
\begin{proof}
Before beginning with the proof note that if $F:\mtr^{m+1}\to\mtc$ is homogeneous of degree $d$ one then has for all $x\in\mtr^{m+1}$, $x\neq0$:
\begin{equation}
\snab{x}^{\mtr^{m+1}} F= d\cdot F(x) \frac{x}{\|x\|^2}+\|x\|^{d-1}\snab{x/\|x\|}^{S^m}(F\lvert_{S^m}).\label{eq: nabla-hom}
\end{equation}
Here we use $\snab{x}^MF$ to denote the gradient in a Riemannian manifold $M$ of a function $F:M\to\mtc$ at a point $x\in M$.

Starting with (i)$\implies$(ii):

A globally defined map $f:S^m\to\mtc$ satisfies $\Delta f=\lambda f$ for some $\lambda\in\mtc$, i.e. is an eigenfunction of the Laplacian, if and only if it is a spherical harmonic, i.e. there is a homogenous harmonic polynomial $F:\mtr^{m+1}\to\mtc$ such that $F\lvert_{S^m}=f$. The homogeneous polynomial $F$ is uniquely determined, denoting its degree by $d$ one has $\lambda=-d(d+m-1)$, see e.g. Chapter III.C.1 of \cite{berger-spectrum} for a textbook treatment.

For a $(\lambda,\mu)$-eigenfamily $\mathcal F$ of functions $S^m\to\mtc$ one then defines:
$$\widetilde{\mathcal F}\defeq \{F\in\mtc[x_1,...,x_{m+1}]\mid F\text{ homogeneous}, \Delta F =0, F\lvert_{S^m}\in\mathcal F\}.$$
In particular  $\lambda=-d(d+m-1)$, where $d$ is the degree of the elements of $\widetilde{\mathcal F}$, and the map $\widetilde{\mathcal F}\to\mathcal F$ induced by the restriction to $S^m$ is a bijection.

Let $F_1, F_2\in\widetilde{\mathcal F}$. By applying (\ref{eq: nabla-hom}) one sees for all $x\in\mtr^{m+1}$, $x\neq0$:
\begin{align*}
(\snab{x}^{\mtr^{m+1}} F_1)^T\,\snab{x}^{\mtr^{m+1}} F_2
&= \|x\|^{2d-2}\ \kappa_{S^m}(F_1\lvert_{S^m}, F_2\lvert_{S^m})(\tfrac x{\|x\|})+\frac{d^2}{\|x\|^2} F_1(x)F_2(x)\\
&= \frac{(\mu+d^2)}{\|x\|^2}F_1(x)F_2(x)
\end{align*}
And so
$$\|x\|^2(\snab{x}^{\mtr^{m+1}} F_1)^T\,\snab{x}^{\mtr^{m+1}} F_2=(\mu +d^2)F_1(x)F_2(x).$$
Since both the left- and right-hand sides are polynomials, one finds that both sides are $0$ or that $F_1F_2$ must have an $\|x\|^2$ factor. Note however that if $m\geq2$ the polynomial $\|x\|^2$ is an irreducible element of $\mtc[x_1,...,x_{m+1}]$, and so if $F_1F_2$ has an $\|x\|^2$ factor then at least one of $F_1$ or $F_2$ also does. However $F_i$ having an $\|x\|^2$ factor is incompatible with the condition $\Delta F_i=0$, whence both $\mu=-d^2$ and $(\nabla^{\mtr^{m+1}} F_1)^T\,\nabla^{\mtr^{m+1}} F_2 =0$ follow.

This verifies that $\widetilde{\mathcal F}$ is a $(0,0)$ eigenfamily and that $\mu=-d^2$, provided $m\geq2$, concluding the direction (i)$\implies$(ii) in this case.

The case $m=1$ follows by a similar argument: Now $\|x\|^2=(x_1+ix_2)(x_1-ix_2)$ is a product of two distinct irreducible elements. Taking $F_1=F_2$ implies that if $F_1^2$ has both factors, then so too must $F_1$, which leads to $\mu=-d^2$. $\mu=-d^2$ then implies $(\nabla^{\mtr^{m+1}} F_1)^T\,\nabla^{\mtr^{m+1}} F_2=0$ for arbitrary $F_1, F_2\in\widetilde{\mathcal F}$.

(ii)$\implies$(i):

Let $\widetilde{\mathcal F}$ be as in the statement of the theorem. As remarked above, the condition $\Delta F =0$ for all $F\in\widetilde{\mathcal F}$ implies $\Delta (F\lvert_{S^m}) = -d(d+m-1) F\lvert_{S^m}=\lambda F\lvert_{S^m}$.

Similarly from $(\nabla^{\mtr^{m+1}}F_1)^T\nabla^{\mtr^{m+1}}F_2=0$ one gets from (\ref{eq: nabla-hom}) that for all $x\in\mtr^{m+1}$, $x\neq0$:
\begin{align*}
0=(\snab{x}^{\mtr^{m+1}} F_1)^T\,\snab{x}^{\mtr^{m+1}} F_2&= \|x\|^{2d-2}\ \kappa_{S^m}(F_1\lvert_{S^m},F_2\lvert_{S^m})(\tfrac{x}{\|x\|})+\frac{d^2}{\|x\|^2} F_1(x)F_2(x)\\
&= \|x\|^{2d-2}\left( \kappa_{S^m}(F_1\lvert_{S^m},F_2\lvert_{S^m})(\tfrac{x}{\|x\|}) + d^2 F_1(\tfrac{x}{\|x\|}) F_2(\tfrac{x}{\|x\|})\right).
\end{align*}
Which implies $\kappa_{S^m}(F_1\lvert_{S^m},F_2\lvert_{S^m}) =- d^2 \, (F_1\lvert_{S^m})\, ( F_2\lvert_{S^m})= \mu\, (F_1\lvert_{S^m})\, ( F_2\lvert_{S^m})$. $\mathcal F$ is then a $(\lambda,\mu)$-eigenfamily.
\end{proof}

From Theorem 2.1 of \cite{ABB-99} it is known that a weakly horizontally conformal map $\mtr^m\to\mtr^n$ for which each component is a polynomial is automatically harmonic. In our case this gives the following simplification:

\begin{corollary}
A family $\mathcal F$ of polynomials $\mtr^m\to\mtc$ is a $(0,0)$-eigenfamily on $\mtr^m$ if and only if for all $x\in\mtr^m$ and $F_1,F_2\in\mathcal F$:
$$(\snab{x} F_1)^T\,\snab{x} F_2 =0.$$
\end{corollary}
In particular a polynomial $F:\mtr^m\to\mtc$ is a harmonic morphism if and only if
\begin{equation}
(\snab{x} F)^T\,\snab{x} F=0\label{eq: iso}
\end{equation}
holds for all $x\in\mtr^m$.

\subsection{Other compact rank one symmetric spaces}\label{subsec: cross}

\pasttheorem{thm: CROSS}
\begin{proof}
In each of the above cases the map $\pi:S^{m(X)}\to X$ is a Riemannian submersion with totally geodesic fibres. For any $F,G\in\mathcal F$ this then implies:
$$\kappa_X( F, G)\circ\pi = \kappa_{S^{m(X)}}( F \circ \pi, G\circ\pi) ,\qquad (\Delta_X F) \circ \pi = \Delta_{S^{m(X)}}(F\circ\pi).$$The statement then follows immediately.
\end{proof}

Together with Theorem \ref{thm: Sn-eigen} this characterises the globally defined eigenfamilies on compact rank one symmetric spaces, with exception of the Cayley plane $\mathds{OP}^2$, in terms of families of polynomial harmonic morphisms. In each case the families that descend to eigenfamilies on the symmetric space are precisely those consisting of $G$-invariant polynomials. Here $G\in\{S^0, S^1, S^3\}$ depending on whether $X$ is a real, complex, or quaternionic projective space.

\begin{example}\label{example: cross}
\begin{enumerate}
\item If $X=\mathds{RP}^m$ then $(\lambda,\mu)$-eigenfamilies correspond to eigenfamilies of homogeneous $\{\pm1\}$-invariant complex valued polynomials on $\mtr^{m+1}$. The condition of $\{\pm1\}$-invariance means that the polynomials must be of even degree. The range of values $(\lambda,\mu)$ for which there exist global $(\lambda,\mu)$-eigenfamilies is $\{(-2d(m-1+2d),-4d^2) \mid d\in\mtn\}$.
\item If $X=\mathds{CP}^m$ then $(\lambda,\mu)$-eigenfamilies correspond to eigenfamilies of homogeneous $U(1)$-invariant complex valued polynomials on $\mtc^{m+1}$. The algebra of $U(1)$-invariant polynomials is generated by $z_i\overline{z_j}$, where $i,j\in\{1,..,m+1\}$. Such polynomials are often called \textit{bi-invariant}. The range of values $(\lambda,\mu)$ for which there exist global $(\lambda,\mu)$-eigenfamilies is $\{(-4d(m+d),-4d^2) \mid d\in\mtn\}$. 
\item If $X=\mathds{HP}^m$ then $(\lambda,\mu)$-eigenfamilies correspond to eigenfamilies of homogeneous $SU(2)$-invariant complex valued polynomials on $\mth^{m+1}$. The range of values $(\lambda,\mu)$ for which there exist global $(\lambda,\mu)$-eigenfamilies is $\{(-4d(2m+1+d),-4d^2) \mid d\in\mtn\}$. 
\end{enumerate}
\end{example}

\section{Relations to holomorphicity}\label{sec: hol}

For two smooth functions $f,g:\mtc^m\to\mtc$ one has for all $x\in\mtc^m$:
\begin{align}
(\nabla_{\!x} f)^T\nabla_{\!x} g &= 2\sum_{j=1}^m \left(\partial_{z_j} f \partial_{\bar{z}_j} g +\partial_{z_j} g \partial_{\bar{z}_j} f\right)(x), \label{eq: hol1}\\
\Delta f\,(x) &= 4\sum_{j=1}^m (\partial_{z_j}\partial_{\bar{z}_j} f)\,(x), \label{eq: hol2}
\end{align}  
where $z_1,...,z_m$ are the standard complex coordinates on $\mtc^m$. In particular a family of holomorphic functions on $\mtc^m$ will trivially be a $(0,0)$-eigenfamily.

Pre-composing a family of holomorphic functions on $\mtc^m$ with an isometry or an orthogonal projection preserves the property that the functions are a $(0,0)$-eigenfamily. As remarked in the introduction, we are mostly interested in eigenfamilies that do \textit{not} arise from families of holomorphic functions in this way. For convenience we will say that families constructed in this manner are \textit{uniformly of complex type}, c.f. Definition \ref{def: complex-type} below.

In Subsection \ref{subsec: complex-type} we give a simple geometric criterion characterising families that are uniformly of complex type, this is Theorem \ref{thm: complex-type}. It turns out to be useful to fix what it means for a function to be \say{holomorphic in some variables}, this is done in in Subsection \ref{subsec: axis-hol} via the notion of an \textit{axis of holomorphy} c.f. Definition \ref{def: axis}. This notion is the source of much of the flexibility one has in generating examples, see also the discussion of Section~\ref{sec: gluing}. 

We use the notion of an axis of holomorphy to show Theorem \ref{thm: 4-ct}, also in Subsection \ref{subsec: axis-hol}. In Subsection \ref{subsec: reduction} we briefly describe how an axis of holomorphy leads to a reduction procedure and use this to show Theorem \ref{thm: 56-reduc}.

\subsection{Functions and families of complex type}\label{subsec: complex-type}

\begin{defn}\begin{enumerate}\label{def: complex-type}
\item A family $\mathcal F$ of functions from $\mtr^m\to\mtc$ said to be \textit{uniformly of complex type} if there is a $k\geq0$, an orthogonal projection $\pi:\mtr^m\to\mtc^k$, and a family $\mathcal G$ of holomorphic maps $\mtc^k\to\mtc$ so that $\mathcal F = \pi^*(\mathcal G) \defeq \{G\circ\pi \mid G\in\mathcal G\}$.
\item A function $F:\mtr^m\to\mtc$ is said to be of \textit{complex type} if the family $\{F\}$ is uniformly of complex type.
\end{enumerate}
\end{defn}

If $m$ is even then a function is of complex type if and only if it is congruent to a holomorphic map, i.e. differs from a holomorphic map only by an isometry of the domain. This is equivalent to the existence of a K\"ahler structure $J$ on $(\mtr^m, g\subs{euc})$ with respect to which the map is holomorphic\footnote{Recall that the complex unit of a K\"ahler structure is parallel.}. If $m$ is odd, then maps of complex type on $\mtr^m$ factor through an even dimensional subspace, on which the above remark applies.

We briefly recall what it means for a subspace of a vector space to be isotropic:

\begin{defn}
A complex linear subspace $V\subseteq \mtr^m\otimes\mtc$ is called \textit{isotropic} if $v^Tw=0$ for all $v,w\in V$.
\end{defn}

\pasttheorem{thm: complex-type}

\begin{proof}
If $\mathcal F$ is uniformly of complex  type then, as remarked before, we may assume the domain to be $\mtc^k$ and each $F\in\mathcal F$ to be holomorphic. Equation (\ref{eq: hol}) is then a direct consequence of the Cauchy-Riemann equations, recall equation (\ref{eq: hol1}).

On the other hand, if equation (\ref{eq: hol}) holds, then $V\defeq\mathrm{span}_\mtc\{\snab{x} F \mid x\in\mtr^m, F\in\mathcal F\}$ is an isotropic subspace of $\mtr^m\otimes\mtc$. Let $z_1,...,z_k$ be an orthonormal basis of $V$ with respect to the usual Hermitian inner product on $\mtr^m\otimes \mtc$. For every $j\in\{1,...,k\}$, define $x_j,y_j\in\mtr^m$ via:
$$z_j = \frac1{\sqrt 2}(x_j+i y_j).$$
Since the vectors $z_j$ are isotropic, orthogonal and of norm $1$, it follows that $x_1,...,x_k,y_1,...,y_k$ constitute an orthonormal system of $\mtr^m$. Define:
$$a:\mtc^k\to\mtr^m,\qquad (\lambda_1,...,\lambda_k)\mapsto \sum_{j=1}^k \mathrm{Re}(\lambda_j)x_j +\sum_{j=1}^k \mathrm{Im}(\lambda_j)y_j,$$
which is then isometric and $\mtr$-linear. We take $\pi=a^*$ to be the adjoint of $a$ and let $\mathcal G = \{F\circ a\mid F\in \mathcal F\}$.

We now check that $F\circ (a a^*)= F$ and that $F\circ a$ is holomorphic for all $F\in\mathcal F$. Note that $a a^*:\mtr^m\to\mtr^m$ is an orthogonal projection to the image of $a$. Suppose that $v\in\mtr^m$ with $v\perp a(\mtc^k)$. Then we have for all $z\in\mtr^m$:
$$\frac{d}{dt}F(z+tv) = (\snab{z+tv} F)^T\cdot v = \sum_{j=1}^k \lambda_j(t) z_j^T v=\sum_{j=1}^k\lambda_j(t)(x_j^T v+ i y_j^T v)=0$$
where we have decomposed $\nabla_{z+tv}F = \sum_j \lambda_j(t) z_j$ for some $\lambda_j(t)\in\mtc$, this is possible since the gradient can be written as a complex linear combination of $z_1,...,z_k$. This implies $F\circ a a^* = F$. For holomorphicity consider $z,v\in\mtc^k$, write $v=(\lambda_1,...,\lambda_k)$, $\snab{a(z)} F=\sum_j \mu_j z_j$, and denote with $J$ the complex unit of $\mtc^k$, then:
\begin{align*}
D_z(F\circ a)(J\cdot v)&= (\nabla_{a(z)} F)^T (\sum_j -\mathrm{Im}(\lambda_j) x_j+\mathrm{Re}(\lambda_j)y_j) = \frac1{\sqrt 2}\sum_j\mu_j (-\mathrm{Im}(\lambda_j)+i\mathrm{Re}(\lambda_j))\\
&= \frac1{\sqrt2}\sum_j i\mu_j \lambda_j = i(\snab{a(z)}F)^Ta(v) = iD_z(F\circ a)(v).
\end{align*}
This establishes the claim.
\end{proof}
\begin{remark}
The short and quick version of the proof of Theorem~\ref{thm: complex-type} is that every isotropic subspace $V\subset\mtr^m{\otimes\mtc}$ is the anti-holomorphic subspace $\{v +iJv\mid v\in \mtr^m, v\perp\ker J\}$ of a degenerate complex unit $J$, i.e. an anti-symmetric linear map $\mtr^m\to\mtr^m$ with $J^2=-\mathds1+p_{\ker J}$, where $p_{\ker J}$ is the orthogonal projection to $\ker J$. If the gradients of a map $F$ are valued in such a subspace then $F$ factors over $(\ker J)^\perp$ and
$$dF(Jv)=(\nabla F)^T(Jv)=-(J \nabla F)^T v = (i \nabla F)^T v =i\, dF(v)$$
for all $v$, whence $F$ is holomorphic with respect to $J$.
\end{remark}
%

\begin{remark}
If $F:\mtr^m\to\mtc$ is a homogenous degree $2$ polynomial, then conditions (\ref{eq: iso}) and (\ref{eq: hol}) coincide. Indeed, if $F(x)=x^TAx$ for $A$ a complex, symmetric $m\times m$ matrix, then:
$$(\nabla_x F)^T\, \nabla_x F =0\ \forall x \iff A^2=0\iff (\nabla_x F)^T\,\nabla_y F =0\ \forall x,y.$$
As a consequence any homogenous degree $2$ polynomial harmonic morphism $\mtr^m\to\mtc$ is of complex type, which is well known\footnote{It follows for example from the classification of homogeneous polynomial degree $2$ harmonic morphisms $\mtr^m\to\mtr^n$, c.f. \cite{ou-wood-96} and \cite{ou-97}.}. A similar calculation shows that any polynomial harmonic morphism of degree $2$ from $\mtr^m$ to $\mtc$, not necessarily homogeneous, is of complex type.\end{remark}

\begin{example}\label{ex: complex-type}
\begin{enumerate}[label=(\roman*)]
\item The polynomials
$$\mtc^4\to\mtc, \ (z,u,v,w)\mapsto zv + uw,\qquad \mtc^4\to\mtc,\ (z,u,v,w)\mapsto z\overline w - u\overline v$$
give a degree $2$ eigenfamily on $\mtc^4$. It is not uniformly of complex type (although the individual polynomials are of complex type).
\item The product of the two polynomials above
$$\mtc^4\to\mtc,\qquad (z,u,v,w)\mapsto z^2vw-u^2\overline{vw}+zu (|w|^2-|v|^2)$$
is a polynomial harmonic morphism of degree $4$, it is not of complex type.
\item As $a,b,c,d$ vary over $\mtc$ the family of maps given by
$$\mtc^4\to\mtc,\qquad (z,u,v,w)\mapsto a(z^2w+zu\overline v)+b(zu\overline w-z^2 v)+c(u^2\overline v+zuw)+d(u^2\overline w-zuv)$$
is an eigenfamily of homogeneous degree $3$ polynomials. This family is not uniformly of complex type, and an element of the family is not of complex type unless $ad-bc=0$.
\item Let $\gamma\in\mtc$, the map
$$\mtc^3\oplus\mtr\mapsto\mtc,\qquad ((z,u,w),t)\mapsto z^2 w + 2\gamma\, zut-\gamma^2 u^2\overline w$$
is a degree $3$ harmonic morphism. It is not of complex type unless $\gamma= 0$.
\end{enumerate}
\end{example}

\begin{remark}
The analog of Theorem~\ref{thm: complex-type} and equation (\ref{eq: hol}) to more general Riemannian manifolds $(M,g)$ would be as follows: Given a family $\mathcal F$ one would consider the equation
\begin{equation}
g_\mtc(P_\gamma \snab{x}^M F, \snab{y}^M G)=0\label{eq: ct-mfd}
\end{equation}
for all $x,y\in M$, $F,G\in M$ and $\gamma$ paths from $x$ to $y$. Here $P_\gamma$ denotes parallel transport. This is equivalent to all parallel translates of gradients from $\mathcal F$ generating a parallel isotropic subspace of $TM\otimes\mtc$. By de Rahm's Theorem the universal cover of $M$ is then a Riemannian product of a K\"ahler manifold and another Riemannian manifold. The family $\mathcal F$ factors over the K\"ahler factor, on which it is a family of holomorphic functions.
\end{remark}

\subsection{Axis of holomorphy}\label{subsec: axis-hol}
In this section we briefly define what it means for a family of functions to be holomorphic in some variables. The given characterisation is geometric, i.e. it does not depend on the choice of a complex structure. We then show Theorem \ref{thm: 4-ct}.
\begin{defn}\label{def: axis}
\begin{enumerate}
\item Let $\mathcal F$ be a family of functions $\mtr^m\to\mtc$. A vector subspace $V\subseteq\mtr^m$ is said to be a \textit{uniform axis of holomorphy} of $\mathcal F$ if the family
$$\{F(x+\cdot):V\to\mtc, v\mapsto F(x+v)\mid x\in\mtr^m, F\in\mathcal F\}$$
is uniformly of complex type.
\item Let $F:\mtr^m\to\mtc$ be a function. A vector subspace $V\subseteq \mtr^m$ is said to be \textit{an axis of holomorphy} of $F$ if $V$ is a uniform axis of holomorphy of $\{F\}$.
\end{enumerate}
\end{defn}

\begin{example}
For all families listed in Example \ref{ex: complex-type} the $\mtc^2$ subspace (of the respective domains) spanned by the variables $z$ and $u$ is a uniform axis of holomorphy.
\end{example}

\begin{remark}
The author does not know \textit{any} examples of polynomial harmonic morphisms $\mtr^m\to\mtc$ without an axis of holomorphy. Numerics indicate that generically the degree $3$ or $4$ homogeneous polynomial harmonic morphisms on $\mtr^6$ have such an axis. In particular these polynomials are generically of complex type by Theorem~\ref{thm: 56-reduc}. In my opinion the question of whether or not an axis must exist is probably the most approachable question when attempting to classify polynomial harmonic morphisms to $\mtc$. It is natural to ask:
\end{remark}

\begin{question}\label{conjecture: axis}
Let $\mathcal F$ be a $(0,0)$-eigenfamily on $\mtr^m$ consisting of homogeneous polynomials. Does $\mathcal F$ admit a non-trivial uniform axis of holomorphy?
\end{question}

We now provide some simple criterions for recognising such an axis. As an immediate consequence of Theorem \ref{thm: complex-type} one sees:

\begin{proposition}
Let $\mathcal F$ be a family of smooth functions from $\mtr^m$ to $\mtc$, and $V\subseteq \mtr^m$ a vector subspace. Denote with $p_V$ the orthogonal projection to $V$. Then $V$ is a uniform axis of holomorphy of $\mathcal F$ if and only if $(\snab{x}F_1)^Tp_V\snab{y} F_2=0$ for all $x,y\in\mtr^m$, $F_1,F_2\in\mathcal F$.
\end{proposition}

In particular a one-dimensional subspace $V\subseteq \mtr^m$ is a uniform axis of holomorphy of a family $\mathcal F$ if and only if $p_V\snab{x}F=0$ for all $x\in\mtr^m$, $F\in\mathcal F$. This is equivalent to $F$ factoring over the projection to $V^\perp$ for all $F\in\mathcal F$. The case of a $2$-dimensional axis is more useful and interesting:

\begin{lemma}\label{lemma: dim2-axis}
Let $V\subseteq\mtr^m$, $\dim(V)=2$ and let $\mathcal F$ be a family of functions $\mtr^m\to\mtc$. Then $V$ is a uniform axis of holomorphy of $\mathcal F$ if and only if there exists an orthonormal basis $\{u,v\}$ of $V$ so that any of the following equivalent conditions hold:
\begin{enumerate}[label=(\roman*)]
\item For all $x\in\mtr^m, F\in\mathcal F$ one has $(\snab{x} F)^T(u+iv) = 0$.
\item For all $F\in\mathcal F$ there is a $\lambda_F:\mtr^m\to\mtc$ so that for all $x\in\mtr^m$ one has
$$p_V(\snab{x} F) = \lambda_F(x)(u+iv).$$
\end{enumerate}
\end{lemma}
\begin{proof}
Note that (ii) clearly implies both (i) and that $V$ is an axis of holomorphy.

On the other hand if $V$ is an axis of holomorphy then either $p_V(\snab{x} F)=0$ for all $x,F$ and then (i) and (ii) follow or there are $y, Q$ with $p_V(\snab{y} Q)\neq0$. Taking $\widetilde u,\widetilde v$ to be the real and imaginary parts of $p_V(\snab{y} Q)$ respectively, one gets from $(\snab{y} Q)^T p_V\snab{y} Q=0$ that they have the same norm and are orthogonal to each other. Then $(\snab{x} F)^Tp_V\snab{y} Q =0$ for all $x, F$ implies (i) for $u=\frac{\widetilde u}{\|\widetilde u\|}, v=\frac{\widetilde v}{\|\widetilde v\|}$.

It remains to show that (i) implies (ii). But this is obvious, since $u+iv, u-iv$ are a complex basis of $V\otimes\mtc$, so there are $\lambda_F, \widetilde \lambda_F:\mtr^m\to\mtc$ with
$$p_V(\snab{x} F)=\lambda_F(x) (u+iv)+\widetilde\lambda_F(x)(u-iv).$$
If $\widetilde\lambda_F(x)\neq0$ for any $x,F$, then (i) fails for this pair $x,F$. Thus $\widetilde\lambda_F=0$ for all $F\in\mathcal F$ and (ii) follows. 
\end{proof}

A result of Baird and Wood states that every globally defined harmonic morphism $\mtr^3\to\mtc$ is of complex type, c.f. Theorem 4.1 in \cite{baird-wood-88}. Wood \cite{wood-92} has shown that the same for harmonic morphisms $\mtr^4\to\mtc$, see e.g. Theorem 7.11.6 in \cite{baird-wood-book}. We extend these results to $(0,0)$-eigenfamilies, using the notion of an axis of holomorphy and Lemma \ref{lemma: dim2-axis}.

\pasttheorem{thm: 4-ct}
\begin{proof}
It suffices to prove the statement for $m=4$. For $F\in\mathcal F$ define $V_F\defeq \spn_\mtc\{\snab{x}F\mid x\in\mtr^4\}$. Then, since $F$ is of complex type, $V_F$ is an isotropic subspace of $\mtr^4\otimes\mtc$ and so $\dim_\mtc(V_F)\in\{0,1,2\}$. Without loss of generality we assume $\dim_\mtc(V_F)\neq0$ for all $F\in\mathcal F$.

We first show that if $\mathcal F$ admits a uniform axis of holomorphy of real dimension $2$ that it is then uniformly of complex type.

By Lemma \ref{lemma: dim2-axis} the existence of such an axis is equivalent to a non-zero isotropic vector $w\in\mtr^4\otimes\mtc$ so that $(\snab{x}F)^T\,w=0$ for all $x\in\mtr^4, F\in\mathcal F$. This $w$ annihilates all $V_F$ and so $V_F+\spn_\mtc\{w\}$ is isotropic for all $F\in\mathcal F$, it is then of dimension $\leq 2$. For any $F\in\mathcal F$ there is then an isotropic vector $w(F)\perp w$ (possibly $w(F)=0$) so that $V_F+\spn_\mtc\{w\}=\spn_\mtc\{w,w(F)\}$. For any $F\in\mathcal F$ there are then functions $\lambda_F,\mu_F:\mtr^4\to\mtc$ for which:
$$\snab{x}F = \lambda_F(x) w + \mu_F(x) w(F).$$
The functions $\mu_F$ can be chosen to be of complex type (since if $w(F)\neq0$ one has $\mu_F(x) = \frac1{\|w(F)\|}(\snab{x}F)^T\,\overline{w(F)}$). In particular each such function is either non-zero on an open dense set or constant zero. Using the fact that $\spn_\mtc\{w,w(F)\}$ is isotropic one gets for all $x\in\mtr^4$ and $F,G\in\mathcal F$ that:
$$(\snab{x}F)^T\,\snab{x}G = \mu_F(x)\mu_G(x)\ w(F)^Tw(G) =0,$$
implying that whenever both of $\mu_F,\mu_G$ are not the zero function that $w(F)^Tw(G)=0$. One then sees:
$$(\snab{x}F)^T\,\snab{y}G = \mu_F(x)\mu_G(y)\, w(F)^Tw(G) =0$$
for all $x,y\in\mtr^4, F,G\in\mathcal F$ and $\mathcal F$ is uniformly of complex type.

For the second step we show that if there is a $F\in\mathcal F$ so that $\dim(V_F)=1$, that $V_F$ is a uniform axis of holomorphy for $\mathcal F$:

There is then a non-zero isotropic vector $w\in\mtr^4\otimes\mtc$ and a function $\lambda:\mtr^4\to\mtc$ so that $\snab{x}F=\lambda(x) w$. As before $\lambda$ is non-zero on an open dense set, so $(\snab{x}G)^T\,\snab{x}F=0$ implies $(\snab{x}G)^T\,w=0$ for all $x\in\mtr^4$, $G\in\mathcal F$. $V_F$ is then a uniform axis of holomorphy for $\mathcal F$ by Lemma \ref{lemma: dim2-axis}.

In the final step we suppose $\dim_\mtc(V_F)=2$ for all $F\in\mathcal F$. We then suppose that $\mathcal F$ is not uniformly of complex type, i.e. that there exist $x,y\in\mtr^4$, $F,G\in\mathcal F$ for which:
$$(\snab{x}F)^T\,\snab{y}G\neq0.$$
This will lead to a contradiction. We choose an orthonormal basis $w_1,w_2$ of $V_F$, then there are functions $\lambda_1,\lambda_2,\mu_1,\mu_2,\mu_3,\mu_4:\mtr^4\to\mtc$ so that:
$$\snab{x}F=\lambda_1(x)w_1+\lambda_2(x)w_2,\quad \snab{x}G=\mu_1(x)\overline{w_1}+\mu_2(x)\overline{w_2} + \mu_3(x){w_1}+\mu_4(x) w_2.$$
The condition $(\snab{x}F)^T\,\snab{x}G=0$ is then equivalent to $\frac{\lambda_1(x)}{\lambda_2(x)}=-\frac{\mu_2(x)}{\mu_1(x)}$ whenever both sides are defined (which we may assume to be almost everywhere).

Since $\lambda_i(x) = (\snab{x}F)^T\,\overline{w_i}$ the quotient $\frac{\lambda_1(x)}{\lambda_2(x)}$ is holomorphic with respect to the same complex structure that makes $F$ holomorphic. In particular its gradient (which is not constant zero) will lie in $V_F$. Similarly the gradient of $-\frac{\mu_2(x)}{\mu_1(x)}$ lies in $V_G$. Since these two functions are however equal one has that $V_F\cap V_G\neq\{0\}$; and so $\{F,G\}$ admits a uniform axis of holomorphy of (real) dimension $2$. By the first step of the proof $\{F,G\}$ is then uniformly of complex type, contradicting the assumption that there are $x,y\in\mtr^4$ so that $(\snab{x}F)^T\,\snab{y}G\neq0$.
\end{proof}

\subsection{Reduction along axis}\label{subsec: reduction}
In this section we remark on a reduction principle for homogeneous polynomial harmonic morphisms that admit an axis of holomorphy (of dimension at least $2$), see Proposition \ref{prop: reduction-axis}. The reduction decreases the dimension of the domain while preserving the property of being a harmonic morphism, at the cost of potentially losing homogeneity.

Since $(0,0)$-eigenfamilies in dimension $\leq4$ are well understood by Theorem \ref{thm: 4-ct}, this yields Theorem \ref{thm: 56-reduc} and Proposition \ref{prop: 9-reduc}.

\begin{defn}
Let $m,d\in\mtn$ and let $F:\mtc\oplus\mtr^m\to\mtc$ be a homogeneous degree $d$ polynomial that is holomorphic in the $\mtc$ factor. Define \textit{the reduction of $F$ along $\mtc$} by:
$$\rho(F):\mtr^m\to\mtc,\qquad x\mapsto F(1,x),$$
which is a polynomial of degree $\leq d$, not necessarily homogeneous.
\end{defn}
\begin{proposition}\label{prop: reduction-axis}
Let $m,d\in\mtn$ and let $\mathcal F$ be a family of homogeneous degree $d$ polynomials from $\mtc\oplus\mtr^m$ to $\mtc$ that are holomorphic in the $\mtc$ factor. Then $\mathcal F$ is an eigenfamily if and only if $\{\rho(F)\mid F\in\mathcal F\}$ is.
\end{proposition}
\begin{proof}
For $F,G\in\mathcal F$ we write with respect to the decomposition $\mtc\oplus\mtr^m$:
$$F(z,x)=\sum_{k=0}^dz^{d-k}F_{k}(x),\qquad G(z,x)=\sum_{k=0}^dz^{d-k}G_{k}(z,x)$$
where $F_k,G_k$ are homogeneous degree $k$ polynomials $\mtr^m\to\mtc$. In particular $\rho(F)(x)=\sum_{k=0}^dF_k(x)$. Then $\mathcal F$ is an eigenfamily if and only if:
\begin{equation}
(\snab{(z,x)} F)^T\, \snab{(z,x)}G=\sum_{k=2}^{2d} z^{2d-k}\sum_{i,j:i+j=k}(\snab{x}F_i)^T\, \snab{x} G_j = 0\label{eq: red-1}
\end{equation}
for all $(z,x)$ and $F,G\in\mathcal F$. On the other hand $\{\rho(F)\mid F\in\mathcal F\}$ is an eigenfamily if and only if
\begin{equation}
(\snab{x}\rho(F))^T\, (\snab{x}\rho(G))=\sum_{k=2}^{2d}\ \sum_{i,j:i+j=k}(\snab{x}F_i)^T\, \snab{x}G_j=0\label{eq: red-2}
\end{equation}
for all $x\in\mtr^m$ and $F,G\in\mathcal F$. By comparing the degrees in $x$ one sees that equations (\ref{eq: red-1}), (\ref{eq: red-2}) are equivalent. The proposition then follows.
%
\end{proof}

Combining Theorem \ref{thm: 4-ct} with the reduction procedure then immediately yields:

\pasttheorem{thm: 56-reduc}
\begin{proof}
Identifying $\mtr^m\cong\mtc\oplus\mtr^{m-2}$ we may then assume the axis of holomorphy to be the $\mtc$ factor. Reducing along this axis as in Proposition \ref{prop: reduction-axis} then results in a polynomial eigenfamily on $\mtr^{m-2}$, which is then of complex type by Theorem \ref{thm: 4-ct}. For each $F\in\mathcal F$ one may write $\rho(F)=\sum_k F_k$ as a sum of homogeneous polynomials, all of different degrees. Since $\{\rho(F)\mid F\in\mathcal F\}$ is uniformly of complex type it follows that $\{F_k \mid k\in\mtn, F\in\mathcal F\}$ is also uniformly of complex type.

Undoing the reduction then implies that $\mathcal F$ is uniformly of complex type.
\end{proof}

\begin{remark}
The statement of Theorem \ref{thm: 56-reduc} does not extend to non-homogeneous polynomials, as the example
$$\mtc^2\oplus\mtr\to\mtc,\qquad ((z,w),t)\mapsto z^2w+ 2\gamma z t - \gamma^2 \overline w$$
(due to Ababou, Baird, and Brossard \cite{ABB-99}) shows, here $\gamma\in\mtc$. Similarly it doesn't extend to domain $\mtr^7$, as demonstrated by inflating the previous example:
$$\mtc^3\oplus\mtr\to\mtc,\qquad ((z,u,w),t)\mapsto z^2w+2\gamma zu t-\gamma^2 u^2\overline w.$$
\end{remark}

By the reduction procedure Proposition \ref{prop: reduction-axis}, homogeneous polynomial harmonic morphisms $\mtr^m\to\mtc$ without an axis of holomorphy contain in some sense information that appears for the first time in dimension $m$. Another way in which a harmonic morphism $F:\mtr^m\to\mtc$ might originate from a combination of lower dimensional data is if there is a decomposition $\mtr^m=V\oplus V^\perp$ with respect to which $F$ is a \textit{harmonic morphism in each variable separately}. This means that for all $(x_0,y_0)\in V\oplus V^\perp$ the maps
$$V\to \mtc, \quad x\mapsto F(x,y_0),\qquad V^\perp\mapsto \mtc, \quad y\mapsto F(x_0,y),$$
are harmonic morphisms.

The content of the next proposition is that, in low dimensions, polynomials without an axis of holomorphy can never be written as a harmonic morphism in each variable separately.
\begin{proposition}\label{prop: 9-reduc}
Let $m\leq 9$ and suppose $F:\mtr^m\to\mtc$ is a homogeneous polynomial harmonic morphism not admitting a non-trivial axis of holomorphy. Then there is no non-trivial decomposition $\mtr^m\cong V\oplus V^\perp$ with respect to which $F$ is a harmonic morphism in each variable separately. 
\end{proposition}
\begin{proof}
Suppose $\mtr^m=V\oplus V^\perp$ is a decomposition with respect to which $F$ is a harmonic morphism in each variable separately. Without loss of generality we assume $\dim(V)\in\{1,2,3,4\}$. We then find for $k\in\{0,..,\deg(F)\}$ homogeneous polynomials $A_k:V\to\mtc$, $B_k:V^\perp\to\mtc$ of degree $k$ so that for $(x,y)\in V\oplus V^\perp$:
$$F(x,y)=\sum_{k=0}^{\deg(F)} A_k(x) B_{\deg(F)-k}(y).$$
Without loss of generality we may assume that none of the maps $B_0,..., B_{\deg(P)}$ are identically zero, then there is a $y_0\in V^\perp$ so that $B_k(y_0)\neq0$ holds simultaneously for all $k$. Since $F$ is a harmonic morphism in both variables separately one gets that the map
$$V\to\mtc,\qquad x\mapsto \sum_{k=0}^{\deg(F)} A_k(x) B_{\deg(P)-k}(y_0)$$
is a harmonic morphism. Since $\dim(V)\leq 4$ it is of complex type. Then the summands $\{A_1,...,A_{\deg(F)}\}$ (which are all of different degrees) are uniformly of complex type. It follows that $V$ is an axis of holomorphy of $F$, which was not permitted.
\end{proof}

\section{Eigenfamilies of degree $2$}\label{sec: deg2}

In this section we discuss eigenfamilies consisting of homogeneous polynomial harmonic morphisms of degree $2$.  We begin by showing that any such family must have a uniform axis of holomorphy of dimension at least $2$.

Subsections \ref{subsec: deg2-statement} and \ref{subsec: deg2-proof} then deal with the classification of pairs of homogeneous degree $2$ polynomials $\{F_1, F_2\}$ that constitute an eigenfamily, see Theorem \ref{thm: deg2-eigenpairs}.

\pasttheorem{thm: deg2-axis}
\begin{proof}
Let $A_1,.., A_k$ be symmetric matrices so that the maps $x\mapsto x^TA_jx$, $j\in\{1,...,k\}$ form a basis of $\spn_\mtc(\mathcal F)$. Note that $\mathcal F$ being an eigenfamily is equivalent to
$$A_i A_j + A_j A_i =0$$
for all $i,j\in\{1,...,k\}$, including $i=j$.

Let $A_{i_1}\cdots A_{i_\ell}\neq0$ be so that $A_j(A_{i_1}\cdots A_{i_\ell})=0$ for all $j\in\{1,...,k\}$. By the anti-commutator relation above such a product exists. If $\ell\neq0$ then the range of $A_{i_1}\cdots A_{i_\ell}$ is non-zero and consists of isotropic vectors that annihilate all of the $A_j$, by Lemma \ref{lemma: dim2-axis} $\mathcal F$ then admits an axis of holomorphy of real dimension $2$. If $\ell=0$ then all of the $A_j$ vanish and the theorem follows anyway.
%
\end{proof}
\subsection{Degree $2$ eigenfamilies with $2$ elements}\label{subsec: deg2-statement}

\begin{defn}
Let $\mathcal F$ be a family of smooth functions $\mtr^m\to\mtc$. A uniform axis of holomorphy $V$ of $\mathcal F$ is called \textit{maximal} if it is not contained in any other uniform axis of holomorphy.
\end{defn}

It is clear that any family of functions $\mtr^m\to\mtc$ admits a maximal uniform axis of holomorphy, and that the dimension of this axis is always the same. 

For economy of notation we will assume all families to be \textit{full}:

\begin{defn}
A family $\mathcal F$ of functions $\mtr^m\to\mtc$ is called \textit{full} if there is no proper subspace $V\subset\mtr^m$ so that $F= F\circ p_V$ for all $F\in \mathcal F$, where $p_V:\mtr^m\to\mtr^m$ denotes the orthogonal projection to $V$.
\end{defn}

In this section we provide a classification of all pairs $F_1, F_2:\mtr^m\to\mtc$ of homogenous degree $2$ polynomials so that $\{F_1,F_2\}$ is an eigenfamily. To this end we introduce some notation:

\begin{defn}\label{def: deg2-data}
\begin{enumerate}
\item We call a triple $(n,k,\delta)$ of natural numbers a \textit{subspace type} if $n\geq k$, $k$ is even, and $\delta\in\{0,1\}$.
\item We call a triple $(P_1,P_2, A)$ \textit{polynomial data} of a subspace type $(n,k,\delta)$ if $P_1,P_2\in\mtc[z_1,...,z_n]$ are homogenous complex polynomials of degree $2$ and $A$ is a complex $n\times k$ matrix of rank $k$.
\item We call a triple $(Y,C,v)$ \textit{twisting data} of a subspace type $(n,k,\delta)$ if $Y, C$ are anti-symmetric $k\times k$ matrices, $Y$ is invertible, and if $v\in\mtc^k$ with $v=0$ if and only if $\delta=0$.
\end{enumerate}
\end{defn}

\pasttheorem{thm: deg2-eigenpairs}

\begin{example}The following are minimal non-trivial examples of such eigenpairs:
\begin{enumerate}[label=(\roman*)]
\item $F_1, F_2:\mtc^4\to\mtc$ given by
$$F_1(z,u,v,w)= zv+uw,\qquad  F_2(z,u,v,w)=z\overline w - z\overline v.$$
Here
$$(n,k,\delta)=(2,2,0),\quad (P_1,P_2,A)=(0,0,\begin{pmatrix}1&0\\0&1\end{pmatrix}),\quad (Y,C,v)=(\begin{pmatrix}0&1\\-1&0\end{pmatrix},0,0).$$
\item $F_1, F_2:\mtc^4\oplus\mtr\to\mtc$ given by
$$F_1(z,u,v,w,t)= zv+uw,\qquad F_2(z,u,v,w,t)= z(\overline w +w +2it)-u\overline v.$$
Here $\delta =1$ and $v=\begin{pmatrix}2\\0\end{pmatrix}$, the other data are as in the previous example.
\end{enumerate}
\end{example}
\begin{remark}
More generally if $M_1,...,M_k$ are symmetric complex $m\times m$ matrices one notes that the the maps $x\mapsto x^TM_i x$ form a degree $2$ $(0,0)$-eigenfamily iff and only if
$$\{M_i,M_j\}=0\ \forall i,j.$$
This is the same data as a representation of the exterior algebra $\rho:\Lambda(\mtc^k)\to \End(\mtr^m)$ for which the generators $\Lambda^1(\mtc^k)$ are sent to symmetric matrices. This is superficially similar to the case for quadratic harmonic morphisms $\mtr^m\to\mtr^n$, which can be classified via the theory of orthogonal representations of Clifford algebras, which is particularly simple. However unlike the Clifford algebras the exterior algebra $\Lambda(\mtc^k)$ is \emph{wild} if $k\geq3$ (since it contains a wild quotient, cf. e.g. \cite{ringel-74}). Since the representation theory of wild algebras is generally regarded as an impossible problem one is unlikely to profit from this approach.
\end{remark}
\begin{remark}
The previous remark does not mean that the classification of e.g. degree $2$ homogeneous $(0,0)$-eigentriples $\{P,Q,O\}$ should be regarded as impossible. As an example the free algebra of two generators is wild, but its representations are easily parametrised by two matrices. It is the problem of classifying the indecomposable representations that is hard.
\end{remark}
%

\subsection{Proof of Theorem \ref{thm: deg2-eigenpairs}}\label{subsec: deg2-proof}

Let $F_1, F_2:\mtr^m\to\mtc$ be two homogenous polynomials of degree $2$ so that $\{F_1,F_2\}$ is full. Let $V$ be a maximal axis of holomorphy for $\{F_1,F_2\}$. Without loss of generality we may apply an isometry to the domain to identify $\mtr^m\cong \mtc^n\oplus \mtr^a$ so that $V$ is identified with the $\mtc^n$ factor. Here $a=m-2n$.

With respect to this decomposition the homogenous degree $2$ polynomials $F_1, F_2$ must be of the form:
\begin{align*}
F_1( (z_1,...,z_n), (x_1,...,x_a)) = P_1(z)+\sum_i z_i\, \xi_i^T x + x^T B_1x,\\
F_2((z_1,...,z_n), (x_1,...,x_a)) = P_2(z)+\sum_i z_i\, \eta_i^Tx+ x^TB_2x.
\end{align*}
Here we introduced $P_1, P_2\in\mtc[z_1,...,z_n]$ two homogenous degree $2$ polynomials, $\xi_i,\eta_i \in\mtr^a\otimes\mtc$ for $i\in\{1,...,n\}$ and $B_1, B_2$ two symmetric complex $a\times a$ matrices. 

The conditions that  $\{F_1, F_2\}$ is a $(0,0)$-eigenfamily become:
\begin{align}
\xi_i^T\xi_j=0=\eta_i^T\eta_j\ \forall i,j,& &B_1^2=0=B_2^2,&  & B_1\xi_i=0=B_2\eta_i\ \forall i,\label{eq: deg2-hm}\\
\xi_i^T\eta_j+\xi_j^T\eta_i=0 \ \forall i,j,& &B_1B_2+B_2B_1=0,& & B_1 \eta_i + B_2\xi_i=0\ \forall i. \label{eq: deg2-family}
\end{align}
Here equations (\ref{eq: deg2-hm}) are the equations that $F_1,F_2$ are $(0,0)$-eigenfunctions, and (\ref{eq: deg2-family}) are the additional conditions necessary for $\{F_1,F_2\}$ to form a $(0,0)$-eigenfamily.
%

For convenience we record the following argument, which we will use often:
\begin{lemma}\label{lemma: deg2-isotropic}
Let $x\in\mtr^a\otimes\mtc$ be isotropic so that
$$x^T\xi_i=0=x^T\eta_i\ \forall i,\qquad B_1x=0=B_2x,$$
then $x=0$.
\end{lemma}
\begin{proof}
Since the gradients of $F_1$ are spanned by $\xi_i$, the range of $B_1$, and the $\mtc^m$ factor it follows that $x$ annihilates every gradient of $F_1$. Similarly it annihilates every gradient of $F_2$. By Lemma~\ref{lemma: dim2-axis} $V_x\defeq\spn_\mtr(\{\mathrm{Re}(x),\mathrm{Im}(x)\})\subseteq \mtr^a\otimes\mtc$ is a uniform axis of holomorphy for $\{F_1,F_2\}$. Clearly $\mtc^n\oplus V_x$ is then a uniform axis for $\{F_1,F_2\}$, if $x\neq0$ this contradicts the maximality of the $\mtc^n$ factor.
\end{proof}

\begin{lemma}\label{lemma: deg2-nonholsquare}
If $B_1, B_2$ are as above, then
$B_1=B_2=0$.
\end{lemma}
\begin{proof}
We first show $B_1\eta_i=0$ for all $i$. Note that for all $j$:
$$\xi_j^TB_1\eta_i=0,\quad \eta_j^TB_1\eta_i=-\eta_j^TB_2\xi_i=0,\quad B_1B_1\eta_i=0, \quad B_2B_1\eta_i=-B_1B_2\eta_i=0,$$
where we used equations (\ref{eq: deg2-hm}), (\ref{eq: deg2-family}) and symmetry of $B_1$. Since $B_1\eta_i$ is isotropic (recall $B_1^2=0$) we apply Lemma~\ref{lemma: deg2-isotropic} and find $B_1\eta_i=0$.

Next let $v\in\mtr^a$, then $B_2B_1v$ is isotropic. We get from equations (\ref{eq: deg2-hm}), (\ref{eq: deg2-family}):
$$\xi_i^TB_2B_1v=0,\qquad \eta_i^T B_2B_1v=0,\qquad B_1B_2B_1v=0,\qquad B_2B_2B_1v=0,$$
whence $B_2B_1v=0$ by Lemma~\ref{lemma: deg2-isotropic}.

Note that $B_1v$ is also istropic. We had just shown $B_2B_1v=0$. Again:
$$\xi_i^TB_1v=0,\qquad \eta_i^TB_1v=0,\qquad B_1B_1v=0.$$
Where this time we also needed $B_1\eta_i=0$. Lemma~\ref{lemma: deg2-isotropic} then implies $B_1v=0$, since $v$ was arbitrary one finds $B_1=0$. In the same way one shows $B_2v=0$ for all $v\in\mtr^a$.
%
%
%
\end{proof}

In what follows we let $H_1$ denote the subspace of $\mtr^a\otimes\mtc$ spanned by the $\xi_i$, let $H_2$ denote the subspace spanned by the $\eta_i$. We let $k\defeq\dim_\mtc(H_1)$, which is equal to $\dim_\mtc(H_2)$ as follows from the next lemma:

\begin{lemma}\label{lemma: subspace-iso}
$H_1, H_2$ are isotropic subspaces of $\mtr^a\otimes\mtc$ and the map
$$\varphi: H_1\to H_2,\qquad \sum_i \lambda_i \xi_i\mapsto \sum_i \lambda_i \eta_i$$
is a well defined isomorphism of complex vector spaces.
\end{lemma}
\begin{proof}
Isotropy follows immediately from $\xi_i^T\xi_j=0=\eta_i^T\eta_j$. In order to check that $\varphi$ is both well defined and an isomorphism it is enough to check that
$$\sum_i \lambda_i \xi_i =0 \iff \sum_i\lambda_i\eta_i=0.$$
Suppose that $\sum_i\lambda_i \xi_i=0$, then:
$$\xi_j^T\left(\sum_i\lambda_i\eta_i\right) = \sum_i\lambda_i \xi_j^T\eta_i = -\left(\sum_i\lambda_i\xi_i\right)^T\eta_j =0.$$
Additionally $\eta_j^T\left(\sum_i\lambda_i\eta_i\right)=0$ for all $j$ and so the isotropic vector $\sum_i\lambda_i\eta_i$ vanishes by Lemma~\ref{lemma: deg2-isotropic}.

The other implication follows by applying the same steps.
\end{proof}

\begin{lemma}
The map $\varphi$ has the property that for all $x,y\in H_1$:
\begin{equation}
x^T\varphi(y) = -\varphi(x)^Ty, \label{eq: phi-antisymm}
\end{equation}
\end{lemma}
\begin{proof}
This follows immediately from $\xi_i^T\eta_j = -\eta_j^T\xi_i$.
\end{proof}

Let $e_1,...,e_k$ denote an orthonormal basis of $H_1$ with respect to the standard Euclidean inner product $(x,y)\mapsto \overline x^T y$. Let $d_1,...,d_\delta\in\mtr^a$ so that
$$e_1,....,e_k,\overline e_1,...,\overline e_k, d_1,...,d_\delta$$
is an orthonormal basis of $\mtr^a\otimes\mtc$ with respect to the same scalar product, choosing such a basis is possible since $H_1$ is isotropic.

\begin{defn}
Let $X,Y$ denote the complex $k\times k$ matrices and $v$ the complex $k\times\delta$ matrix defined by:
$$\varphi(e_i) = \sum_j X_{ij}e_j+\sum_j Y_{ij}\overline e_j+\sum_\ell v_{i\ell}d_\ell.$$
If $\delta=0$ we use the convention $v=0$.
\end{defn}

\begin{lemma}\label{lemma: Y}
$Y$ is anti-symmetric and invertible.
\end{lemma}
\begin{proof}
By (\ref{eq: phi-antisymm}) one has
$$e_i^T\varphi(e_j)=Y_{ij}=-e_j^T\varphi(e_i)=-Y_{ji}$$
and so $Y$ is anti-symmetric.

Suppose that $Y$ were not invertible and let $\lambda_1,...,\lambda_k\in\mtc$ be so that $\sum_j Y_{ij}\lambda_j=0$ for all $i$. Then $$e_i^T\varphi(\sum_j \lambda_j e_j)=0$$
for all $i$ and $\varphi(\sum_j\lambda_je_j)$ annihilates $H_1$, in particular all of the $\xi_i$. Since it is an element of $H_2$ and so annihilates all of the $\eta_i$. By Lemma~\ref{lemma: deg2-isotropic} it vanishes and then $\lambda_j=0$ for all $j$.

It follows that $Y$ is injective, since it is a square matrix it is invertible.
\end{proof}

\begin{remark}
It follows that $k=\dim_\mtc(H_1)$ is even.
\end{remark}

\begin{lemma}
$\delta\in\{0,1\}$. If $\delta=1$ then $v$ is a non-zero element of $\mtc^k$.
\end{lemma}
\begin{proof}
Note that for all $j,\ell$ one has:
$$\varphi(e_j)^T\left(d_\ell+\sum_i(v^TY^{-1})_{\ell i}e_i\right)= v_{j\ell}+ \sum_i (v^TY^{-1})_{\ell i}Y_{ji}= v_{j\ell}+(v^TY^{-1}Y^T)_{\ell j}=v_{j\ell}-v_{j\ell}=0.$$
If $\delta\geq2$ one then checks that:
$$\left(d_1+\sum_i (vY^{-1})_{1i}e_i\right)+i\left(d_2+\sum_i (vY^{-1})_{2i}e_i\right)$$
is isotropic and annihilates both $H_1$ and $H_2$. Lemma~\ref{lemma: deg2-isotropic} then implies that it is identically zero, which contradicts $d_1, d_2$ being linearly independent. We conclude $\delta\leq1$. 

In the case that $\delta=1$ we identify the $k\times1$ matrix $v$ with an element of $\mtc^k$. The case $\delta=1, v=0$ is not possible since then $d_1$ is a real vector annihilating both $H_1$ and $H_2$, contradicting fullness of $\{F_1,F_2\}$.
\end{proof}

\begin{remark}
From the previous results it follows that the numbers $(n,k,\delta)$ defined above satisfy the conditions of being a subspace type. The basis $e_1,...,e_k,\overline{e_1},...,\overline{e_k},d_\delta$ yields an isomorphism of real vector spaces:
$$\mtc^k\oplus\mtr^\delta\to\mtr^a,\qquad ((w_1,...,w_k),t)\mapsto \sum_i \overline{w_i} e_i + \sum_i w_i\overline{e_i}+td_\delta.$$
\end{remark}

Denote with $A$ the $n\times k$ matrix defined by the equations
$$\xi_i=\sum_j A_{ij}e_j $$
for all $i$. Note that $A$ has rank $k$ since the $\xi_i$ span the $k$-dimensional space $H_1$.

\begin{remark}$(P_1, P_2, A)$ satisfy the conditions of being polynomial data for $(n,k,\delta)$.  Additionally the steps so far show:
$$F_1((z_1,...,z_n), (w_1,...,w_k),t)=P_1(z)+\sum_{ij} z_i A_{ij} w_j.$$
\end{remark}

\begin{lemma}\label{lemma: C-v}
There is an anti-symmetric complex $k\times k$ matrix $C$ so that $X= (\frac12 vv^T+C)Y^{-1}$.
\end{lemma}
\begin{proof}
Note that:
$$\varphi(e_i)^T\varphi(e_j) = (XY^T)_{ij}+(XY^T)_{ji}+(vv^T)_{ij}.$$
Recall that $H_2$ is an isotropic vector space, so $\varphi(e_i)^T\varphi(e_j)=0$ for all $i,j$. Together with $Y^T=-Y$ this equation then becomes
$$(XY)+(XY)^T=vv^T,$$
which implies the lemma.
\end{proof}
\begin{remark}
Lemmas \ref{lemma: Y} and \ref{lemma: C-v} then show that the triple $(Y,C,v)$ satisfies the conditions of being twisting data for $(n,k,\delta)$.
\end{remark}
Since $\eta_i=\varphi(\xi_i) = \sum_{j}A_{ij}\varphi(e_j)$ one then gets:
$$F_2((z_1,...,z_n),(w_1,...,w_k),t) = P_2(z)+\sum_{ij}z_i A_{ij} (\sum_l X_{jl} w_l + \sum_l Y_{jl}\overline{w_l} + v_j t)$$
This finishes the characterisation of full eigenfamilies $\{F_1,F_2\}$ of degree $2$ homogenous polynomials.

It remains to see that any choice of subspace type and polynomial as well as twisting data determines such an eigenfamily. But this is an elementary calculation that we skip.\qed

\section{Constructions of polynomial eigenfamilies}\label{sec: constructions}

This section details some ways in which eigenfamilies of homogeneous polynomial harmonic morphims from $\mtr^m$ to $\mtc$ may be constructed.

\subsection{Eigenfamilies induced by polynomial harmonic morphisms $\mtr^m\to\mtr^n$}\label{subsec: from-higher}

Harmonic morphisms from $\mtr^m$ to $\mtr^n$ that are polynomial in each component have been studied before. While Eells and Yiu \cite{eells-yiu-95} find all such polynomials that restrict to a map of spheres $S^{m-1}\to S^{n-1}$, and all homogeneous polynomial harmonic morphisms $\mtr^m\to\mtr^n$ of degree $2$ are classified by the works of Ou and Wood \cite{ou-97, ou-wood-96}, not so much is known about higher degree polynomials. Ababou, Baird, and Brossard \cite{ABB-99} derive restrictions on the dimensions of domain and codomain of a homogeneous polynomial harmonic morphism, and show that the symbol at a critical point of a weakly horizontal map between two Riemannian manifolds $M^m\to N^n$ gives a homogeneous polynomial harmonic morphism $\mtr^{m}\to\mtr^{n}$. See chapter 5 of \cite{baird-wood-book} for a unified treatment.

One checks the following elementary property of harmonic morphisms with codomain $\mtr^n$, see e.g. \cite{baird-wood-book}:
\begin{lemma}
A map $P:\mtr^m\to\mtr^n$ is a harmonic morphism if and only if for all $j,k\in\{1,...,n\}$, $j\neq k$ the map $P_j+iP_k:\mtr^m\to\mtc$ is a harmonic morphism.
\end{lemma}

It then immediately follows:
\begin{proposition}
If $P:\mtr^m\to\mtr^n$ is a harmonic morphism then the family of functions
$$\mathcal E(P)\defeq \left\{ P_{2k-1}+i P_{2k}\ \middle|\  k\in\{1,...,\lfloor\frac n2\rfloor\} \right\}$$
is a $(0,0)$-eigenfamily.
\end{proposition}

In particular if $P:\mtr^m\to\mtr^n$ is a homogeneous polynomial harmonic morphism then $\mathcal E(P)$ is an eigenfamily of homogeneous polynomials.

\begin{example}
Let $P:\mth^3\to\mth, (q_1,q_2,q_3)\mapsto q_1\cdot q_2\cdot q_3$ be the multiplication of three quaternions. This is a harmonic morphism and upon identifying $\mth^3\cong \mtc^6$, $\mathcal E(P)$ is a degree $3$ eigenfamily consisting of the following two maps:
\begin{align*}
\mtc^6\to\mtc,\quad (z_1,z_2,u_1,u_2,w_1,w_2)\mapsto z_1(u_1w_1-u_2\overline{w_2})-z_2(\overline{u_1w_2}+\overline{u_2}w_1),\\
\mtc^6\to\mtc,\quad (z_1,z_2,u_1,u_2,w_1,w_2)\mapsto z_1(u_1w_2+u_2\overline{w_1})+z_2(\overline{u_1w_1}-\overline{u_2}w_2).
\end{align*}
\end{example}

Note that if $B\in O(n)$ and $P:\mtr^m\to\mtr^n$ is a harmonic morphism then $B\circ P$ is also a harmonic morphism. It may be interesting to investigate when $\mathcal E(P)$ and $\mathcal E(B\circ P)$ give different eigenfamilies. For convenience we define:
\begin{defn}
Let $\mathcal F_1, \mathcal F_2$ be two eigenfamilies of maps $\mtr^m\to\mtc$, we say that $\mathcal F_1,\mathcal F_2$ are congruent if there is an isometry $\Phi:\mtr^m\to\mtr^m$ so that
$$\spn_\mtc\mathcal F_1 = \spn_\mtc\Phi^*(\mathcal F_2),$$
where $\Phi^*(\mathcal F_2)\defeq\{F\circ \Phi \mid F\in\mathcal F_2\}$.
\end{defn}

Using the classification of homogeneous degree $2$ polynomial harmonic morphisms $\mtr^m\to\mtr^n$ one can make the following observations:

\begin{theorem}\label{thm: deg2-higher}
Let $P:\mtr^m\to\mtr^n$ be a non-zero homogeneous degree $2$ polynomial harmonic morphism.
\begin{enumerate}
\item If $n\geq4$ then $\mathcal E(P)$ is not uniformly of complex type.
\item For any $B\in O(n)$ the families $\mathcal E(P)$ and $\mathcal E(B\circ P)$ are congruent.
\end{enumerate}
\end{theorem}

\begin{proof}
We show point 2 first. From the classification of such maps via weighted direct sums of irreducible Clifford systems, one sees that if $n\not\equiv 1\,\text{mod}\,4$, that there is an isometry $A\in O(m)$ so that $B\circ P=P\circ A$, implying that $\mathcal E(P)$ and $\mathcal E(B\circ P)$ are congruent. This follows since any two irreducible representations of the Clifford algebra $C_n$ by symmetric matrices are \textit{algebraically equivalent} in this case, see e.g. Chapter 5 \cite{baird-wood-book}.

For the case $n\equiv1\,\text{mod}\, 4$ let $\pi:\mtr^n\to\mtr^{n-1}$ denote the projection to the first $n-1$ components. Then, since $n$ is odd, $\mathcal E(P)=\mathcal E(\pi\circ P)$ and $\mathcal E(B\circ P)=\mathcal E(\pi\circ B\circ P)$. Decomposing $P$ and $B\circ P$ as a weighted sum of irreducible Clifford systems. One notes that the weights appearing in both decompositions must be the same (counted with multiplicity), although the irreducible representation associated to any weight may be different. It is elementary to check that the weights appearing in the Clifford decompositions of $\pi\circ P$ and $\pi\circ B\circ P$ will be the same as those appearing in $P$ and $B\circ P$, using again that any two irreducible Clifford systems on $\mtr^{n-1}$ are algebraically equivalent one sees that $\mathcal E(\pi\circ P)$ and $\mathcal E(\pi\circ B\circ P)$ are congruent.

For the first point let $\pi_4:\mtr^n\to\mtr^4$ denote the projection to the first $4$ components, then $\pi_4\circ P:\mtr^m\to\mtr^4$ is a harmonic morphism. Since $\mathcal E(\pi_4\circ P)$ is a subfamily of $\mathcal E(P)$, it suffices to check that the statement in the case $n=4$. As before $P$ is a weighted direct sum of irreducible Clifford systems, so it suffices to check the statement for one such summand. In this case this amounts to checking that the map $\mathds H\oplus\mathds H\to\mathds H$ given by multiplication of two quaternions does not induce a family of complex type. Since the resulting family is in fact congruent to the one spanned by Example~\ref{ex: complex-type} (i), the remaining step follows.
\end{proof}
This observation for degree $2$ leads to the natural questions:

\begin{question}Let $P:\mtr^m\to\mtr^n$ be a homogeneous polynomial harmonic morphism of degree $\deg(P)\geq2$, and suppose $n\geq4$.
\begin{enumerate}
\item Is $\mathcal E(P)$ not of uniformly complex type?
\item Are the families $\mathcal E(P)$ and $\mathcal E(B\circ P)$ congruent for all $B\in O(n)$?
\end{enumerate}
\end{question}

\subsection{Complex defects}
\begin{defn}
Let $F:\mtr^m\to\mtc$ and let $y\in\mtr^m$, the \textit{complex defect of $F$ at $y$} is the map
$$[F]_y:\mtr^m\to\mtc,\qquad x\mapsto (\snab{x} F)\snab{y} F.$$
\end{defn}
\begin{remark}
$[F]_y=0$ for all $y\in\mtr^m$ if and only if $F$ is of complex type.
\end{remark}

The motivation for considering this definition is the following statement, which follows e.g. from the proof of Theorem 2.1 in \cite{ABB-99}:

\begin{proposition}
Suppose $F:\mtr^m\to\mtc$ is a homogeneous polynomial harmonic morphism of degree $3$. Then for all $y\in\mtr^m$ the defects $[F]_y$ are homogeneous polynomial harmonic morphisms of degree $2$.
\end{proposition}

It is not necessarily the case that the defect of a higher degree harmonic morphism is again a harmonic morphism, let alone an eigenfamily. Nevertheless it works out in the following example:

\begin{example}
Consider
$$P:\mtc^4\to\mtc,\qquad (z,u,v,w)\mapsto z^2 vw - u^2 \overline{vw}+zu(|v|^2-|w|^2),$$
which is a degree $4$ harmonic morphism. Then the complex defects $\{[P]_y\mid y\in\mtc^4\}$ constitute an eigenfamily of degree $3$ harmonic morphisms and $\spn_\mtc\{[P]_y\mid y\in\mtc^4\}$ is given by all functions of the form
$$a(z^2w+zu\overline v)+b(zu\overline w-z^2v)+c(u^2\overline v+zuw)+d(u^2\overline w-zuv),$$
where $a,b,c,d\in \mtc$. Note that such a function is of complex type if and only if $ad-bc=0$.
\end{example}

\subsection{Gluing and augmenting along an axis of holomorphy}\label{sec: gluing}

We now give two observations on how an axis of holomorphy can be used to expand eigenfamilies and to generate new eigenfamilies. The proofs are trivial and so we omit them:

\begin{proposition}\label{prop: adjoin-hol}
Let $\mathcal F$ be $(0,0)$-eigenfamily of functions from $\mtc^k\oplus\mtr^m\to\mtc$ which are holomorphic in the $\mtc^k$ factor. For any family $\mathcal G$ of holomorphic functions $\mtc^k\to\mtc$ and $\pi:\mtc^k\oplus\mtr^m\to\mtc^k$ the orthogonal projection one has that
$$\spn_\mtc(\mathcal F)+\spn_\mtc(\pi^* \mathcal G)$$
is a $(0,0)$-eigenfamily.
\end{proposition}
\begin{example}
Let $\mathcal F$ be the family on $\mtc^4$ spanned by
$$\mtc^4\to\mtc, \quad (z,u,v,w)\mapsto a(z^2w+zu\overline v)+b(zu\overline w-z^2v)+c(u^2\overline v+zuw)+d(u^2\overline w-zuv)$$
where $a,b,c,d\in\mtc$. Denote by $\mtc[z,u]_3$ the homogeneous degree $3$ polynomials in $z,u$, then
$$\mathcal F \oplus\mtc[z,u]_3$$
is a $(0,0)$-eigenfamily consisting of homogeneous degree $3$ polynomials.
\end{example}

\begin{proposition}\label{prop: gluing}
Let $\mathcal F$ be a $(0,0)$-eigenfamily $\mtc^k\oplus\mtr^{m_1}\to\mtc$ and $\mathcal G$ a $(0,0)$-eigenfamily $\mtc^k\oplus\mtr^{m_2}\to\mtc$, both holomorphic in the $\mtc^k$ factors. Then the following is a $(0,0)$-eigenfamily on $\mtc^k\oplus\mtr^{m_1}\oplus\mtr^{m_2}$:
$$\{(z,x,y)\mapsto P(z,x)\mid P\in\mathcal F\}\cup\{(z,x,y)\mapsto Q(z,y)\mid Q\in\mathcal G\}.$$
\end{proposition}
\begin{example}
Let $P,Q:\mtc^2\oplus\mtc^2\to\mtc$ be given by $P(z,u,v,w)=zv+uw$, $Q(z,u,v,w)=z\overline w -u\overline v$, $\mathcal F=\mathcal G=\{P,Q\}$. Gluing these as in Proposition~\ref{prop: gluing} gives an eigenfamily on $\mtc^6$ spanned by the functions:
\begin{align*}
P_1(z,u,v_1w_1,v_2,w_2)=zv_1+uw_1,& & Q_1(z,u,v_1w_1,v_2,w_2) = z\overline{w_1}-u\overline{v_1},\\
P_2(z,u,v_1w_1,v_2,w_2)=zv_2+uw_2,& & Q_2(z,u,v_1w_1,v_2,w_2) = z\overline{w_2}-u\overline{v_2}.
\end{align*}
\end{example}

\begin{remark}
Taking powers of $(0,0)$-eigenfamilies and Propositions~\ref{prop: adjoin-hol}, \ref{prop: gluing} are methods with which one can easily generate large amounts of $(0,0)$-eigenfamilies from previously known ones. It seems however quite difficult to find new $(0,0)$-eigenfamilies that do not arise in this way from known families. This suggests that the question of classifying polynomial $(0,0)$-eigenfamilies has some rigidity, which should be most visible in low dimensions and degrees.
\end{remark}

%
%
%

\bibliographystyle{amsplain}
\bibliography{my}

\end{document}